\documentclass[twoside,11pt,reqno]{amsart}

\usepackage{amsmath}
\usepackage{amssymb}
\usepackage{amscd}
\usepackage{mathrsfs}
\usepackage{epic}
\usepackage{latexsym}
\usepackage{tikz}
\usepackage{mathrsfs}
\usepackage{cite}
\usepackage{hyperref}
\usepackage{mathtools} 
\usepackage{pb-diagram}
\usepackage{tikz-cd}
\usepackage[all]{xy}
\usepackage{xcolor}
\usepackage{enumitem}
\usepackage{xargs} 
\usepackage{xparse} 
\usetikzlibrary{cd} 
\usepackage{tensor} 
\usepackage{xpatch} 
\allowdisplaybreaks
\numberwithin{equation}{section}
\newtheorem{Proposition}[equation]{Proposition}
\newtheorem{Lemma}[equation]{Lemma}
\newtheorem{Theorem}[equation]{Theorem}
\newtheorem{Corollary}[equation]{Corollary}
\theoremstyle{definition} % makes all of the theorem environments which follow appear in \rm

\let\ss\scriptstyle 
\let\sss\scriptscriptstyle 
\newcommand{\mtiny}[1]{\mbox{\tiny$#1$}}
\newcommand{\mtinier}[1]{\mbox{\tinier$#1$}}
\newcommand{\mminiscule}[1]{ \mbox{\miniscule$#1$}}
\usepackage{lmodern}
\makeatletter
\ifcase \@ptsize \relax% 10pt
\newcommand{\miniscule}{\@setfontsize\miniscule{4}{5}}% \tiny: 5/6
\or% 11pt
\newcommand{\miniscule}{\@setfontsize\miniscule{4}{5}}% \tiny: 6/7
\or% 12pt
\newcommand{\miniscule}{\@setfontsize\miniscule{5}{6}}% \tiny: 6/7
\fi
\makeatother
\makeatletter
\ifcase \@ptsize \relax% 10pt
\newcommand{\tinier}{\@setfontsize\tinier{4}{5}}% \tiny: 5/6
\or% 11pt
\newcommand{\tinier}{\@setfontsize\tinier{5}{6}}% \tiny: 6/7
\or% 12pt
\newcommand{\tinier}{\@setfontsize\tinier{5}{6}}% \tiny: 6/7
\fi
\makeatother
\NewDocumentCommand\smallsup{mom}{\IfNoValueTF{#2}{{{{#1}^{{\sss{#3}}}}}\vphantom{#1}}{{{{#1}_{#2}^{{\sss{#3}}}}}\vphantom{#1}}}
\NewDocumentCommand\tinysup{mom}{\IfNoValueTF{#2}{{{{#1}^{{\mtiny{#3}}}}}\vphantom{#1}}{{{{#1}_{#2}^{{\mtiny{#3}}}}}\vphantom{#1}}}
\NewDocumentCommand\tiniersup{mom}{\IfNoValueTF{#2}{{{{#1}^{{\mtinier{#3}}}}}\vphantom{#1}}{{{{#1}_{#2}^{{\mtinier{#3}}}}}\vphantom{#1}}}
\NewDocumentCommand\tiniestsup{mom}{\IfNoValueTF{#2}{{{{#1}^{{\mminiscule{#3}}}}}\vphantom{#1}}{{{{#1}_{#2}^{{\mminiscule{#3}}}}}\vphantom{#1}}}
\NewDocumentCommand\lowsmallsup{mom}{\IfNoValueTF{#2}{\def\arraystretch{0.2}#1\begin{array}{@{}l@{}}{\sss #3}\\ {\phantom{\sss #1}}\end{array}\vphantom{#1}}{\def\arraystretch{0.1}#1\begin{array}{@{}l@{}}{\sss #3}\\ {\vphantom{#1}^{\ss #2}}\end{array}\vphantom{#1}}}
\NewDocumentCommand\lowsmallestsup{mom}{\IfNoValueTF{#2}{\def\arraystretch{0.1}#1\begin{array}{@{}l@{}}{\mminiscule{#3}}\\ {\phantom{\mminiscule{#1}}}\end{array}\vphantom{#1}}{\def\arraystretch{0.1}#1\begin{array}{@{}l@{}}{\mminiscule{#3}}\\ {\vphantom{#1}^{\ss #2}}\end{array}\vphantom{#1}}}
\NewDocumentCommand\lowsup{mom}{\IfNoValueTF{#2}{\def\arraystretch{0.1}#1\begin{array}{@{}l@{}}{ \ss{#3}}\\ {\phantom{\ss{#1}}}\end{array}\vphantom{#1}}{\def\arraystretch{0.1}#1\begin{array}{@{}l@{}}{\ss{#3}}\\ {\vphantom{#1}^{\ss #2}}\end{array}\vphantom{#1}}}
\NewDocumentCommand\smallsub{mmo}{\IfNoValueTF{#3}{{{{#1}_{{\sss{#2}}}}}\vphantom{#1}}{{{{#1}^{#3}_{{\sss{#2}}}}}\vphantom{#1}}}
\NewDocumentCommand\tinysub{mmo}{\IfNoValueTF{#3}{{{{#1}_{{\mtiny{#2}}}}}\vphantom{#1}}{{{{#1}^{#3}_{{\mtiny{#2}}}}}\vphantom{#1}}}
\NewDocumentCommand\tiniersub{mmo}{\IfNoValueTF{#3}{{{{#1}_{{\mtinier{#2}}}}}\vphantom{#1}}{{{{#1}^{#3}_{{\mtinier{#2}}}}}\vphantom{#1}}}
\NewDocumentCommand\tiniestsub{mom}{\IfNoValueTF{#3}{{{{#1}_{{\mminiscule{#2}}}}}\vphantom{#1}}{{{{#1}^{#3}^{{\mminiscule{#2}}}}}\vphantom{#1}}}
\NewDocumentCommand\lowsmallsub{mmo}{\IfNoValueTF{#3}{\def\arraystretch{0.3}#1\begin{array}{@{}l@{}}{\phantom{\sss #1}}\\ {\sss #2}\end{array}\vphantom{#1}}{\def\arraystretch{0.1}#1\begin{array}{@{}l@{}}{{\ss #3}}\\ {\sss #2}\end{array}\vphantom{#1}}}
\NewDocumentCommand\lowsmallestsub{mmo}{\IfNoValueTF{#3}{\def\arraystretch{0.1}#1\hspace{-0.2ex}\begin{array}{@{}l@{}}{\phantom{\mminiscule #1}}\\ {\mminiscule #2}\end{array}\vphantom{#1}\hspace{-0.4ex}}{\def\arraystretch{0.1}#1\begin{array}{@{}l@{}}{{\ss #3}}\\ {\mminiscule #2}\end{array}\vphantom{#1}}}
\newcommand{\Hom}{\operatorname{Hom}}
\newcommand{\Ext}{\operatorname{Ext}}
\newcommand{\End}{\operatorname{End}}

\newcommand{\ind}{\operatorname{ind}}

\newcommand{\im}{{\mathrm{im}\,}}

\newcommand{\rad}{\operatorname{rad}}
\newcommand{\soc}{\operatorname{soc}}
\newcommand{\St}{\operatorname{St}}

\DeclareFontFamily{U}{mathx}{\hyphenchar\font45}
\DeclareFontShape{U}{mathx}{m}{n}{
<5> <6> <7> <8> <9> <10>
<10.95> <12> <14.4> <17.28> <20.74> <24.88>
mathx10
}{}
\DeclareSymbolFont{mathx}{U}{mathx}{m}{n}
\DeclareFontSubstitution{U}{mathx}{m}{n}
\let\bigoplus\relax
\DeclareMathSymbol{\bigoplus}{1}{mathx}{"C0}
\begin{document}

\date{\today}

\thanks{}

\title[Humphreys-Verma Conjecture in rank $2$]{{\bf On the Humphreys-Verma Conjecture for semisimple  algebraic groups of rank $2$}}

\author{\sc Stephen Donkin}
\address{Department of Mathematics\\ University of York\\ York YO10 5DD, U.K.} 
\email{stephen.donkin@york.ac.uk}

\author{\sc Haralampos Geranios}
\address{Department of Mathematics\\ National and Kapodistrian University of Athens \\ Athens 157 84, Greece.} 
\email{hgeranios@math.uoa.gr}

\thanks{The second author gratefully acknowledges the support of The Royal Society through the URF research grant URF\textbackslash R\textbackslash221047}

\begin{abstract}
Let $G$ be a connected, semisimple, simply connected algebraic group over an algebraically closed field of positive characteristic. For each restricted dominant weight $\lambda$, there is the  associated principal indecomposable $G_1$-module $Q_1(\lambda)$, where $G_1$ is the first infinitesimal subgroup of $G$. The assertion that, for every such $\lambda$, there exists a $G$-module whose restriction to $G_1$ is isomorphic to $Q_1(\lambda)$ is known as the Humphreys--Verma Conjecture.
For groups of rank $2$, it was shown in \cite{BNPS1} that the Humphreys--Verma Conjecture holds in all cases except one, namely when $G$ is of type $G_2$, the characteristic is $2$, and $\lambda=0$. This case remained completely open. Moreover, in every previously resolved case, the module $Q_1(\lambda)$ could be realized as the restriction of a suitable tilting module. However, in \cite{BNPS2} it was shown that $Q_1(0)$ for $G_2$ in characteristic $2$ cannot arise as the restriction of a tilting module, thereby providing the first counterexample to a conjecture of the first author.
In this paper, we construct a $G$-module whose restriction to $G_1$ is $Q_1(0)$, thereby establishing the Humphreys--Verma Conjecture in the last remaining rank $2$ case. Our construction provides the first known example of a $G$-structure on a principal indecomposable $G_1$-module that does not arise from a tilting module. This reveals a new phenomenon in the study of the Humphreys--Verma Conjecture and suggests new directions for understanding $G$-structures on principal indecomposable $G_1$-modules.
\end{abstract}
\maketitle

\section{Introduction}\label{sec:intro}
Let $\mathbb F$ be an algebraically closed field.  Let $G$ be a connected, semisimple, simply connected algebraic group over $\mathbb F$. We view $G$ as a Chevalley group. Let $T$ be the maximal torus, $\Phi$ the root system, with system of positive roots $\Phi^+$, and $B$ the negative Borel subgroup arising from the Chevalley construction. We assume from now on that the characteristic of ${\mathbb F}$ is $p>0$.  We write $G_1$ for the first infinitesimal subgroup of $G$. By a $G$-structure on a $G_1$-module $X$ we mean a $G$-module whose restriction to $G_1$ is $X$.

We write  $X^+(T)$ for the set of dominant weights for $G$, associated to the maximal torus $T$, and  $X_1\subset X^+(T)$ for the set of restricted dominant weights. For $\lambda\in X^+(T)$ we write $L(\lambda)$ for the corresponding irreducible $G$-module associated to $\lambda$. For $\lambda\in X_1$, the $G_1$-modules $L(\lambda)|_{G_1}$ form a complete set of pairwise non-isomorphic irreducible $G_1$-modules. It follows that every irreducible $G_1$-module admits a $G$-structure, a fact already established in Curtis's pioneering work in the 1960s. For $\lambda\in X_1$ we write $Q_1(\lambda)$ for the injective hull of $L(\lambda)$, as a $G_1$-module. Motivated by,  and in analogy with, Curtis' results, Humphreys and Verma  considered in the 1970s the problem of establishing the existence of  a $G$-structure on the $G_1$-modules $Q_1(\lambda)$ \cite{HV}. The assertion that such a {\mbox{$G$-structure}} exists in general became known as the Humphreys-Verma conjecture. A decade later, the first author proposed several related conjectures including one indicating where such a $G$-structure may be found: precisely that it should be the indecomposable tilting module of a certain highest weight. For $p\geq 2h-2$ this holds, essentially by Jantzen's construction \cite{J3}.  Over time that this should hold in general became known as the Tilting Module Conjecture.
In  \cite{BNPS2}  the authors show that the Tilting Module Conjecture fails when $G$ is the algebraic group of type $G_2$, $\mathbb{F}$ is a field of characteristic $2$, and $\lambda=0$. The primary goal of the present paper is to construct a $G$-structure on $Q_1(0)$, thereby establishing the Humphreys--Verma Conjecture in this case. For all other rank $2$ groups, the Tilting Module Conjecture is known to hold, thanks to the work of Bendel, Nakano, Pillen, and Sobaje \cite{BNPS1}. Consequently, the Humphreys--Verma Conjecture is now known to be true for all rank $2$ groups.
   
 The paper is organized as follows. In Section~\ref{sec:prel}, we establish notation and introduce the basic tools used throughout the paper. In Section~\ref{sec:coh.frob}, we study the cohomology of $G_1$-modules, while Section~\ref{sec:ext.irr} is devoted to the cohomology of $G$-modules. In Section~\ref{sec:sub.in}, we focus on the  structure of certain induced $G$-modules. In Section~\ref{sec:surj.tilt}, we analyse the principal block of the tensor square $\St \otimes \St$ of the Steinberg module $\St$ and construct a surjective $G$-homomorphism from $\St \otimes \St$ onto the tilting module $T(2,1)$. This homomorphism is a key ingredient in the construction of a $G$-structure on $Q_1(0)$, which is carried out in Section~\ref{sec:struct}.     
    
\section{Preliminaries}\label{sec:prel}

For an affine group scheme  $G$ over ${\mathbb F}$ we write ${\rm mod}(G)$ for the category of finite dimensional  $G$-modules. For $V\in{\rm mod}(G)$ we denote by $V^*\in {\rm mod}(G)$ the dual module of $V$ and for a positive integer $m$ we write $V^m$ for the direct sum of $m$ copies of $V$.

We write  $X(T)$ for the character group of $T$. We have that $X(T)$ is equipped with a natural  partial order: for $\lambda,\mu\in X(T)$ we have $\mu\leq \lambda$ if $\lambda-\mu$ is the sum of simple roots. The corresponding group ring ${\mathbb Z}X(T)$ has basis $e(\lambda)$, $\lambda\in X(T)$, whose elements multiply according to the rule $e(\lambda)e(\mu)=e(\lambda+\mu)$.  The Weyl group $W$ of $G$ (with respect to $T$) acts on ${\mathbb Z} X(T)$.  For $\lambda\in X(T)$, we have the element $A(\lambda)={\sum_{w\in W} \rm sgn}(w) e(w\lambda)$, where ${\rm sgn}(w)$ is the sign of the Weyl group element $w\in W$. We write $\rho$ for half the sum of the positive roots. Then $A(\lambda+\rho)$ is divisible by $A(\rho)$ in ${\mathbb Z}X(T)$ and the Weyl character $\chi(\lambda)$ is defined to be the quotient $A(\lambda+\rho)/A(\rho)\in ({\mathbb Z}X(T))^W$.  

A $T$-module $V$ has weight space decomposition $V=\bigoplus_{\lambda\in X(T)} V^\lambda$ and the character of $V$ is defined by ${\rm ch} V=\sum_{\lambda\in X(T)} \dim V^\lambda e(\lambda)$.  For $\lambda\in X(T)$ we have the one dimensional {\mbox{$B$-module}} ${\mathbb F}_\lambda$, on which $T$ acts according to $\lambda$.  We have the induction functor ${\rm ind}\colon{\rm mod}(B)\to {\rm mod}(G)$, and its derived functors $R^i{\ind}_B^G$, $i\geq 0$. We write $X^+(T)\subset X(T)$ for the set of dominant weights.  For $\lambda\in X(T)$,  the induced module ${\rm ind}_B^G {\mathbb F}_\lambda$  is non-zero if and only if $\lambda\in X^+(T)$. For $\lambda\in X^+(T)$ the module  $\nabla(\lambda)={\rm ind}_B^G {\mathbb F}_\lambda$  has simple socle, which we denote  $L(\lambda)$. Moreover, the character of $\nabla(\lambda)$ is $\chi(\lambda)$.

By a good filtration of $V\in {\rm mod}(G)$ we mean a filtration $0\leq V_1\leq  \cdots \leq V_m$ such that, for each $1\leq i\leq m$, the section $V_i/V_{i-1}$ is either $0$ or isomorphic to $\nabla(\lambda^i)$, for some $\lambda^i\in X^+(T)$.  If $V\in {\rm mod}(G)$ admits a good filtration then, for $\lambda\in X^+(T)$, the number of sections  isomorphic to $\nabla(\lambda)$ in a good filtration is independent of the choice of filtration, and will be denoted $(V:\nabla(\lambda))$. We shall use, without further reference,  that the tensor product of modules with a good filtration has a good filtration, \cite[II, 4.21 Proposition]{J1},  and  the acyclicity of modules with a good filtration, in particular that  $H^i(G,\nabla(\lambda)\otimes \nabla(\mu))=0$, for all $\lambda,\mu\in X^+(T)$ and all $i>0$, \cite[II, 4.13  Proposition]{J1}.

For $\lambda\in X^+(T)$ we write $\Delta(\lambda)$ for the Weyl module corresponding to $\lambda$. We have that $\Delta(\lambda)=\nabla(-w_0 \lambda)^*$, where $w_0$ is the longest element of the Weyl group $W$. The Weyl module $\Delta(\lambda)$ has simple head $L(\lambda)$. By a Weyl filtration of $V\in {\rm mod}(G)$ we mean a filtration $0\leq V_1\leq  \cdots \leq V_m$ such that, for each $1\leq i\leq m$, the section $V_i/V_{i-1}$ is either $0$ or isomorphic to $\Delta(\lambda^i)$, for some $\lambda^i\in X^+(T)$.  If $V\in {\rm mod}(G)$ admits a Weyl filtration then, for $\lambda\in X^+(T)$, the number of sections  isomorphic to $\Delta(\lambda)$ in a Weyl filtration is independent of the choice of filtration, and will be denoted $(V:\Delta(\lambda))$. 

For each $\lambda\in X^+(T)$ there is an indecomposable  $G$-module $T(\lambda)$ with unique highest weight $\lambda$ occurring with multiplicity one such that $T(\lambda)$ and the dual module $T(\lambda)^*$ admit a good filtration (equivalently $T(\lambda)$ admits a  good filtration  and a Weyl filtration). The module  $T(\lambda)$  is uniquely determined (up to isomorphism) by these conditions and is called the indecomposable  tilting module of highest weight $\lambda$. We have a surjective $G$-homomorphism from $T(\lambda)$ to $\nabla(\lambda)$ and an embedding of $\Delta(\lambda)$ into $T(\lambda)$. 
   
Let $F:G\to G$ be the Frobenius morphism arising from the Chevalley construction. For $V\in {\rm mod}(G)$ affording representation $\pi:G\to {\rm GL}(V)$ we write   $V^F$ for the $G$-module affording the representation $\pi\circ F$. For $\lambda\in X^+(T)$ the Frobenius twist $\nabla(\lambda)^F$ embeds in $\nabla(p\lambda)$ as the space of $G_1$ fixed points.
  
We shall make extensive use of the $5$-term exact sequence  for a $G$-module $V$:
$$0\to H^1(G/G_1,V^{G_1})\to H^1(G,V)\to H^1(G_1,V)^G\to H^2(G/G_1,V^{G_1})\to H^2(G,V)$$
 
We shall also make use of Brauer's formula  \cite[II, 5.8 Lemma]{J1} and several standard identities involving the characters $\chi(\lambda)$. Namely, for $\lambda\in X^+(T)$ and an element $\sum_{\mu\in X(T)} a_\mu e(\mu)\in ({\mathbb Z} X(T))^W$, one has
 \[\chi(\lambda)\left(\sum_{\mu\in X(T)} a_\mu e(\mu)\right)= \sum_{\mu\in X(T)} a_\mu\,\chi(\lambda+\mu),\]
and
\[\chi(w\cdot\lambda)={\rm sgn}(w)\chi(\lambda)\]
for all $w\in W$ and $\lambda\in X(T)$, where $w\cdot\lambda = w(\lambda+\rho)-\rho$ denotes the dot action of $W$ on $X(T)$ \cite[II, 5.9 (1)]{J1}. 
 We shall also need a couple of simple general lemmas.
 
 \begin{Lemma}\label{lem:bootstrap}
Let $\lambda\in X^+(T)$ and let  $X$ a be  $G$-submodule of the induced $G$-module $\nabla(\lambda)$. For $\mu \in X^+(T)$ with $\mu\leq \lambda$ we have that
$$\Hom_G(L(\mu),\nabla(\lambda)/X)=\Ext^1_G(L(\mu),X).$$
\end{Lemma}

\begin{proof}
We consider the short exact sequence 
$$0\rightarrow X\rightarrow \nabla(\lambda)\rightarrow  \nabla(\lambda)/X\rightarrow 0$$
Applying $\Hom_G(L(\mu),-)$ to the short exact sequence above, we obtain the exact sequence
\begin{align*}
0 &\longrightarrow \Hom_G(L(\mu),X)
   \longrightarrow \Hom_G(L(\mu),\nabla(\lambda))
   \longrightarrow \Hom_G\bigl(L(\mu),\nabla(\lambda)/X\bigr)\\
  &\longrightarrow \Ext^1_G(L(\mu),X)
   \longrightarrow \Ext^1_G(L(\mu),\nabla(\lambda))
\end{align*}
Now, by  \cite[II, 4.13, Remark (1))]{J1} we have that $\Ext^1_G(L(\mu),\nabla(\lambda))=0$, {\mbox{since $\mu\leq\lambda$.}} Moreover, $\Hom_G(L(\mu),X)= \Hom_G(L(\mu),\nabla(\lambda))$, since ${\rm soc}_GX={\rm soc}_G\nabla(\lambda)=L(\lambda)$. Therefore, it follows that $\Hom_G(L(\mu),\nabla(\lambda)/X)=\Ext^1_G(L(\mu),X)$.
\end{proof}

 \begin{Lemma}\label{lem:wei.filt}
Let $\lambda\in X^+(T)$ and $V\in {\rm mod}(G)$. Suppose that $\lambda+\eta+\rho \in X^+(T)$ for every weight $\eta$ of $V$. Then $\nabla(\lambda)\otimes V$ has a good filtration and the multiplicity of $\nabla(\mu)$ in such a filtration is the dimension of the weight space $V^{\tau-\lambda}$. 
 \end{Lemma}
 
 \begin{proof} We have $\nabla(\lambda)\otimes V\cong {\rm ind}_B^G({\mathbb F}_\lambda\otimes V)$, by the tensor identity  \cite[I, Proposition 3.6]{J1}. We regard $V$ as a $B$-module and we choose a composition series $0=V_0 < V_1<\cdots < V_m=V$ for $V$. Thus, for $1\leq i\leq m$, the section  $V_i/V_{i-1}$ is isomorphic to ${\mathbb F}_{\eta_i}$, for some $\eta_i\in X(T)$. 
 
 We have the ascending filtration $0=V_0\otimes {\mathbb F}_\lambda < V_1\otimes {\mathbb F}_\lambda < \cdots < V_m\otimes {\mathbb F}_\lambda=V\otimes {\mathbb F}_\lambda$. For $1\leq j\leq m$ we have $(V_j\otimes {\mathbb F}_\lambda)/(V_{j-1}\otimes {\mathbb F}_\lambda)\cong {\mathbb F}_{\eta_i}$ and hence, by Kempf's vanishing theorem  \cite[II, Proposition 4.5]{J1}, $R^1 {\rm ind}_B^G((V_j\otimes {\mathbb F}_\lambda)/(V_{j-1}\otimes {\mathbb F}_\lambda))=0$. Hence, we also have $R^1{\rm ind}_B^G (V_i\otimes {\mathbb F}_\lambda)=0$. We may identify  ${\rm ind}_B^G (V_i\otimes {\mathbb F}_\lambda)$ with a submodule of ${\rm ind}_B^G (V\otimes {\mathbb F}_\lambda)$ and the quotient  ${\rm ind}_B^G (V_i\otimes {\mathbb F}_\lambda)/{\rm ind}_B^G (V_{i-1}\otimes {\mathbb F}_\lambda)$ with $\ind_B^G(V_i\otimes {\mathbb F}_\lambda)/(V_{i-1}\otimes {\mathbb F}_\lambda)$. The result follows.
 \end{proof}

From now on, and until the end of the paper we specialise  to the group $G$ of Lie type $G_2$, which is the main focus of our attention and take $\mathbb F$ to have characteristic $2$.  We have simple roots $\alpha,\beta$, with $\alpha$ short, and corresponding fundamental dominant weights $\omega_\alpha,\omega_\beta$. For $a,b\in {\mathbb Z}$ we write $(a,b)$ for $a\omega_\alpha+ b\omega_\beta$. Thus $\alpha=(2,-1)$, $\beta=(-3,2)$. We have that $X^+(T)=\{(a,b)\in \mathbb{Z}^2\mid a,b\geq0\}$ and $X_1=\{(a,b)\in X^+(T)\mid a,b\leq 1\}$. We write $\St$ for the Steinberg module for $G$. We have that $\St=\nabla(1,1)=\Delta(1,1)$. The following results are standard and their statements can be found for instance in \cite[\S3.2.1]{A}.

\begin{Lemma}\label{lem:ch.T(2,1)}
For the tilting module $T(2,1)$ we have the following:
\begin{enumerate}[label = (\roman*), font=\normalfont]
\item $\St\otimes L(1,0)=T(2,1)$;
\item $T(2,1)|_{G_1}=Q_1(0,1)$;
\item The character of $T(2,1)$ is $\chi(2,1)+\chi(0,2)+\chi(3,0)+\chi(2,0)+\chi(0,1)$. 
\end{enumerate}
\end{Lemma}

\begin{Lemma}\label{lem:ch.T(1,2)}
For the tilting module $T(1,2)$ we have the following:
\begin{enumerate}[label = (\roman*), font=\normalfont]
\item $\St\otimes L(0,1)=T(1,2)\oplus \St^2$;
\item $T(1,2)|_{G_1}=Q_1(1,0)$;
\item The character of $T(1,2)$ is $\chi(1,2)+\chi(4,0)+\chi(2,1)+\chi(3,0)+\chi(3,0)+\chi(1,0)$. 
\end{enumerate}
\end{Lemma}

\begin{Lemma}\label{lem:ch.stst}
For the tilting module $\St\otimes \St$ we have the following:
\begin{enumerate}[label = (\roman*), font=\normalfont]
\item $\St\otimes \St|_{G_1}=Q_1(0,0)\oplus Q_1(0,1)^2\oplus \St^{16}$;
\item The character of $\St\otimes \St$ is 
\begin{align*}
\chi(2,2) +& \chi(5,0) + \chi(1,2) + \chi(0,3)+ 3\chi(2,1) + 2\chi(4,0) + 2\chi(0,2) + 3\chi(3,0)\\
+& 2\chi(2,0) + 2\chi(0,1) + \chi(1,0) + \chi(0,0)+ 2\chi(3,1) + 2\chi(1,1).
\end{align*}
\end{enumerate}
\end{Lemma}

\section{Cohomology of $G_1$-modules}\label{sec:coh.frob}
Here, we calculate the first cohomology group $H^1(G_1,M)$ for certain $G_1$-modules $M$.  The  following result from \cite[Lemma 3.3]{DS} (see also \cite[\S 5.3 (2)]{J2}) gives a complete list of extensions for simple $G_1$-modules, as $G$-modules. 

\begin{Lemma}\label{lem.coh.simples}
 The following hold:
\begin{enumerate}[label = (\roman*), font=\normalfont]
\item there are no self extensions of simple $G_1$-modules;
\item $H^1(G_1, L(1,0))=T(1,0)^{\rm{F}}$;
\item $H^1(G_1, L(0,1))=\nabla(1,0)^{\rm{F}}$;
\item $H^1(G_1, L(1,0)\otimes L(0,1))=0$.
\end{enumerate}
\end{Lemma}
We note that $T(1,0)$ appears as the unique extension 
$$0\to \nabla(0)\to T(1,0)\to \nabla(1,0)\to 0$$
We now proceed with the $G_1$-cohomology of certain induced $G$-modules.

\begin{Lemma}\label{lem.coh.ind}
 The following hold:
\begin{enumerate}[label = (\roman*), font=\normalfont]
\item $H^1(G_1, \nabla(1,0))=\nabla(1,0)^{\rm{F}}$;
\item $H^1(G_1, \nabla(2,0))=0$.
\end{enumerate}
\end{Lemma}
\begin{proof}
(i) We consider the short exact sequence 
$$0\rightarrow L(1,0)\rightarrow \nabla(1,0)\rightarrow L(0,0)\rightarrow0$$
We have $\nabla(1,0)^G=0$ and by Lemma~\ref{lem.coh.simples}  $H^1(G_1,L(0,0))=0$. Thus, from the long exact sequence of $G_1$-cohomology we obtain the exact sequence
$$0\to L(0,0)\to H^1(G_1,L(1,0))\to H^1(G_1,\nabla(1,0))\to 0.$$
Moreover, by Lemma~\ref{lem.coh.simples} we have $H^1(G_1,L(1,0))=T(1,0)^F$ and the result follows.

(ii) The quotient $\nabla(2,0)/\nabla(1,0)^F$ has composition factors $L(1,0),L(0,1)$ and hence $\nabla(2,0)/\nabla(1,0)^F=L(1,0)\oplus L(0,1)$. The short exact sequence 
$$0\rightarrow \nabla(1,0)^F\rightarrow \nabla(2,0)\rightarrow L(1,0)\oplus L(0,1)\rightarrow 0$$
gives rise to the exact sequence
$$H^1(G_1,\nabla(1,0)^F)\to H^1(G_1,\nabla(2,0)) \to H^1(G_1,L(1,0))\oplus H^1(G_1,L(0,1))$$
By Lemma~\ref{lem.coh.simples} we have $H^1(G_1,\nabla(1,0)^F)=0$ and so $H^1(G_1,\nabla(2,0))$ embeds in 
$$H^1(G_1,L(1,0))\oplus H^1(G_1,L(0,1))=T(1,0)^F\oplus \nabla(1,0)^F,$$ which has $G$-socle $L(0,0)^F\oplus L(1,0)^F$.  If $H^1(G_1,\nabla(2,0))\neq 0$, then we either have $H^1(G_1,\nabla(2,0))^G\neq 0$ or 
$(H^1(G_1,\nabla(2,0))\otimes \nabla(1,0)^F)^G\neq 0$.   Now, $H^1(G_1,\nabla(2,0))^G=0$ by the $5$-term exact sequence (and the acyclicity of $\nabla(2,0)$ as a $G$-module). 
It remains to show that   $(H^1(G_1,\nabla(2,0))\otimes \nabla(1,0)^F)^G=H^1(G_1,\nabla(2,0)\otimes \nabla(1,0)^F)^G$ is $0$. We have $\nabla(2,0)^{G_1}=\nabla(1,0)^F$ and 
 $$H^i(G/G_1,\nabla(1,0)^F\otimes \nabla(1,0)^F)=H^i(G,\nabla(1,0)\otimes \nabla(1,0))=0$$
for all $i>0$. Thus, by the $5$-term exact sequence we obtain that
$$H^1(G_1,\nabla(2,0))\otimes \nabla(1,0)^F)^G=H^1(G,\nabla(2,0)\otimes \nabla(1,0)^F)$$
Now, $\nabla(1,0)^F$ has composition factors $L(2,0)$ and $L(0,0)$ and so  by \cite[II, 4.13, Remark (1))]{J1} we deduce that $H^1(G,\nabla(2,0)\otimes \nabla(1,0)^F)=0$.
\end{proof}

\section{Some Extensions between irreducible $G$-modules}\label{sec:ext.irr}

Here, we calculate the groups $\Ext^1_G(L(\lambda),L(\mu))$, for $\lambda,\mu\in X^+(T)$ with $\lambda,\mu\leq (2,1)$, using the results of Section~\ref{sec:coh.frob}. We first record a lemma on good filtrations.
\begin{Lemma}\label{lem.good.filtrations}
The following hold: 
\begin{enumerate}[label=(\roman*),font=\normalfont]
\item$\nabla(1,0)\otimes \nabla(1,0)$ has a good filtration and its character is $$\chi(2,0)+\chi(0,1)+\chi(1,0)+\chi(0,0);$$
\item $\nabla(1,0)\otimes L(0,1)$ has a good filtration and  its character is $\chi(2,0)+\chi(1,0)+\chi(1,1)$;
\item$\nabla(1,0)\otimes L(1,0)$ has a good filtration and  its character is $\chi(2,0)+\chi(0,1)+\chi(0,0)$.
\end{enumerate}
\end{Lemma}

\begin{proof}
Part (i) and part (ii) are clear and the sections in the filtration follow directly from Brauer's formula. Part (iii) follows from Lemma~\ref{lem:wei.filt}.
\end{proof}

\begin{Lemma}\label{lem.ext.simples}
For $\lambda,\mu\in X^+(T)$ with $\lambda,\mu\leq (2,1)$, the dimensions of the extension groups $\Ext^1_G(L(\lambda),L(\mu))$
are given in the following table:
\[
\renewcommand{\arraystretch}{1.15}
\begin{array}{c|ccccccc}
& (0,0) & (1,0) & (0,1) & (2,0) & (3,0) & (0,2) & (2,1)
\\
\hline
(0,0) &0&1&0&1&0&0&1\\
(1,0) &1&0&0&0&1&0&0\\
(0,1) &0&0&0&1&0&0&1\\
(2,0) &1&0&1&0&1&0&0\\
(3,0) &0&1&0&1&0&1&0\\
(0,2) &0&0&0&0&1&0&1\\
(2,1) &1&0&1&0&0&1&0
\end{array}
\]
\end{Lemma}

\begin{proof}
From \cite[II, 2.12, (1)]{J1} we have that irreducible $G$-modules admit no nontrivial self-extensions and so the diagonal entries are zero. Moreover, from \cite[II, 2.12,(4)]{J1} (or the self  duality of simple modules) we have that
\[
\Ext^1_G(L(\lambda),L(\mu))
=
\Ext^1_G(L(\mu),L(\lambda)),
\]
so it suffices to consider the case $\mu>\lambda$. We use the 5-term exact sequence and Lemma~\ref{lem.coh.simples} and we have the following:

\noindent\textit{Case 1: $\lambda=(0,0)$.} We have
\[
\begin{aligned}
\Ext^1_G(L(0,0),L(1,0)) &= H^1(G,L(1,0)) = H^0(G,T(1,0)) = \mathbb F,\\
\Ext^1_G(L(0,0),L(0,1)) &= H^1(G,L(0,1)) = H^0(G,\nabla(1,0)) = 0,\\
\Ext^1_G(L(0,0),L(2,0)) &= H^1(G,L(1,0)^{\mathrm F}) = H^1(G,L(1,0)) = \mathbb F,\\
\Ext^1_G(L(0,0),L(3,0)) &= H^1(G,L(1,0)\otimes L(1,0)^{\mathrm F})
= H^0(G,T(1,0)\otimes L(1,0)) = 0,\\
\Ext^1_G(L(0,0),L(0,2)) &= H^1(G,L(0,1)^{\mathrm F}) = H^1(G,L(0,1)) = 0,\\
\Ext^1_G(L(0,0),L(2,1)) &= H^1(G,L(0,1)\otimes L(1,0)^{\mathrm F})
= H^0(G,\nabla(1,0)\otimes L(1,0)) = \mathbb F.
\end{aligned}
\]

\medskip

\noindent\textit{Case 2: $\lambda=(1,0)$.} We have
\[
\begin{aligned}
\Ext^1_G(L(1,0),L(0,1))
&= H^1(G,L(1,0)\otimes L(0,1)) = 0,\\
\Ext^1_G(L(1,0),L(2,0))
&= H^1(G,L(1,0)\otimes L(1,0)^{\mathrm F})
= H^0(G,T(1,0)\otimes L(1,0)) = 0,\\
\Ext^1_G(L(1,0),L(3,0))
&= H^1(G,L(1,0)\otimes L(1,0)\otimes L(1,0)^{\mathrm F})
= H^1(G,L(1,0))
= \mathbb F,\\
\Ext^1_G(L(1,0),L(0,2))
&= H^1(G,L(1,0)\otimes L(0,1)^{\mathrm F})
= H^0(G,T(1,0)\otimes L(0,1)) = 0,\\
\Ext^1_G(L(1,0),L(2,1))
&= H^1(G,L(1,0)\otimes L(0,1)\otimes L(1,0)^{\mathrm F}) = 0.
\end{aligned}
\]

\medskip

\noindent\textit{Case 3: $\lambda=(0,1)$.} We have
\[
\begin{aligned}
\Ext^1_G(L(0,1),L(2,0))
&= H^1(G,L(0,1)\otimes L(1,0)^{\mathrm F})
=H^0(G, \nabla(1,0)\otimes L(1,0))=\mathbb F,\\
\Ext^1_G(L(0,1),L(3,0))
&= H^1(G,L(0,1)\otimes L(1,0)\otimes L(1,0)^{\mathrm F})
 = 0,\\
\Ext^1_G(L(0,1),L(0,2))
&= H^1(G,L(0,1)\otimes L(0,1)^{\mathrm F})
=H^0(G, \nabla(1,0)\otimes L(0,1))=0,\\
\Ext^1_G(L(0,1),L(2,1))
&= H^1(G,L(0,1)\otimes L(0,1) \otimes L(1,0)^{\mathrm F})
= H^1(G,L(1,0)) = \mathbb F.
\end{aligned}
\]

\medskip
\noindent\textit{Case 4: $\lambda=(2,0)$.} We have 
\[
\begin{aligned}
\Ext^1_G(L(2,0),L(3,0))
&= H^1(G,L(1,0)\otimes L(1,0)^{\mathrm F}\otimes L(1,0)^{\mathrm F})\\
&= H^0(G,T(1,0)\otimes L(1,0)\otimes L(1,0))
\end{aligned}
\]
Now consider the short exact sequence
\[0\to L(1,0)\to L(1,0)\otimes T(1,0)\to L(1,0)\otimes \nabla(1,0)\to 0\]
By Lemma~\ref{lem.good.filtrations}(iii), we have that $H^0(G,\nabla(1,0)\otimes L(1,0)\otimes L(1,0))=0$ and so
\[\begin{aligned}
\Ext^1_G(L(2,0),L(3,0))
&=H^0(G,T(1,0)\otimes L(1,0)\otimes L(1,0))\\
&=H^0(G,L(1,0)\otimes L(1,0))=\mathbb F,\\
\Ext^1_G(L(2,0),L(0,2))
&= H^1(G,L(1,0)^{\mathrm F}\otimes L(0,1)^{\mathrm F})
= H^1(G,L(1,0)\otimes L(0,1))
=0, \\
\Ext^1_G(L(2,0),L(2,1))
&= H^1(G,L(0,1)\otimes L(1,0)^{\mathrm F}\otimes L(1,0)^{\mathrm F})\\
&= H^0(G,\nabla(1,0)\otimes L(0,1)\otimes L(0,1))
=0,
\end{aligned}
\]
where the last equality follows from Lemma~\ref{lem.good.filtrations}(ii).

\medskip

\noindent\textit{Case 5: $\lambda=(3,0)$.} We have 
\[\begin{aligned}
\Ext^1_G(L(3,0),L(0,2))
&= H^1(G,L(1,0)\otimes L(1,0)^{\mathrm F}\otimes L(0,1)^{\mathrm F})\\
&= H^0(G,T(1,0)\otimes L(1,0)\otimes L(0,1))
\end{aligned}\]
Now, consider the short exact sequence
\[0\to L(0,1)\to L(0,1)\otimes T(1,0)\to L(0,1)\otimes \nabla(1,0)\to 0\]
Since $\Ext^1_G(L(1,0),L(0,1))=0$, we obtain
\[\begin{aligned}
\Ext^1_G(L(3,0),L(0,2))
&= H^0(G,T(1,0)\otimes L(1,0)\otimes L(0,1))\\
&= H^0(G,\nabla(1,0)\otimes L(1,0)\otimes L(0,1))=\mathbb F,
\end{aligned}\]
where the last equality follows from Lemma~\ref{lem.good.filtrations}(ii).
\[\begin{aligned}
\Ext^1_G(L(3,0),L(2,1))
&=H^1(G,L(1,0)\otimes L(0,1)\otimes L(1,0)^{\mathrm F}\otimes L(1,0)^{\mathrm F})
=0.
\end{aligned}\]

\medskip
\noindent\textit{Case 6: $\lambda=(0,2)$.} We have 
\[
\begin{aligned}
\Ext^1_G(L(0,2),L(2,1))
&= H^1(G,L(0,1)\otimes L(1,0)^{\mathrm F}\otimes L(0,1)^{\mathrm F})\\
&= H^0(G,\nabla(1,0)\otimes L(1,0)\otimes L(0,1))
= \mathbb F.
\end{aligned}
\]
where the last equality follows from Lemma~\ref{lem.good.filtrations}(ii).
\end{proof}

\section{The submodule structure of certain induced $G$-modules}\label{sec:sub.in}
In this section we construct filtrations for the induced $G$-modules $\nabla(0,2)$ and $\nabla(2,1)$. The following calculation will prove useful in our considerations.
\begin{Lemma}\label{lem.gcoh}
We have 
$$H^1(G,L(0,1)\otimes \nabla(1,0)\otimes \nabla(1,0)^{\mathrm F})=0.$$
\end{Lemma}

\begin{proof}
From Lemma~\ref{lem.good.filtrations}(ii) the $G$-module  $L(0,1)\otimes \nabla(1,0)$ has a good filtration with sections $\nabla(2,0)$, $\nabla(1,0)$ and ${\rm St}$. Therefore we have that 
$$H^1(G,L(0,1)\otimes \nabla(1,0)\otimes L(1,0)^{\mathrm F})=H^1(G,K\otimes L(1,0)^{\mathrm F}),$$ 
where $K$ is the $G$-submodule of $L(0,1)\otimes \nabla(1,0)$ that lies in the principal block (i.e., the block of ${\rm mod}(G)$ containing $L(0)$). Therefore $K$ has a good filtration with sections $\nabla(2,0)$ and $\nabla(1,0)$.  Now, certainly 
$H^0(G_1,K)=0$. We show that $H^1(G_1,K)$  also vanishes. For this, we consider the short exact sequence  of $G$-modules
$$0\rightarrow \nabla(1,0)\rightarrow K \rightarrow \nabla(2,0) \rightarrow 0$$ 
and we obtain the exact sequence 
\begin{align*}
0\longrightarrow H^0(G_1,\nabla(2,0))
  \longrightarrow  H^1(G_1,\nabla(1,0))
   \longrightarrow H^1(G_1,K)
    \longrightarrow H^1(G_1,\nabla(2,0))
    \end{align*}
 Now, recall that  $H^0(G_1,\nabla(2,0))=\nabla(1,0)^F$. Moreover, $H^1(G_1,\nabla(1,0))=\nabla(1,0)^F$ and $H^1(G_1,\nabla(2,0))=0$, by Lemma~\ref{lem.coh.ind}. It follows that $H^1(G_1,K)=0$. Therefore by the 5-term exact sequence we deduce that 
$$H^1(G,L(0,1)\otimes \nabla(1,0)\otimes L(1,0)^{\mathrm F})=H^1(G,K\otimes L(1,0)^{\mathrm F})=0.$$
\end{proof}

\begin{Lemma}\label{lemma:(0,2)}
The induced module  $\nabla(0,2)$ has  the following  $G$-submodule structure:
\[
%\nabla(0,2)\;=\;
\frac{\dfrac{L(0,1)}
{\dfrac{L(2,0)}{L(0,0)}}}
{\dfrac{L(1,0)}
{\dfrac{L(3,0)}{L(0,2)}}}
\;=\;
\frac{\dfrac{L(0,1)}
{\Delta(1,0)^{\mathrm F}}}
{\dfrac{L(1,0)\otimes \nabla(1,0)^{\mathrm F}}
 {L(0,2)}}
\]
\end{Lemma}

\begin{proof}
This is  may be found for instance in \cite[Appendix C]{Do}. It was also observed independently by H.~Andersen, see \cite[Remark~3.5.2]{BNPS2}. We note that it may also be easily verified directly using Lemma~\ref{lem:bootstrap} together with the extension table in Lemma~\ref{lem.ext.simples}.
\end{proof}

 \begin{Lemma}\label{lemma:(2,1)}
The  $G$-module $\nabla(2,1)$ has a filtration given as follows:
\[
%\nabla(2,1)\;=\;
\frac{\dfrac{\nabla(0,2)}{L(1,0)}}
{{\dfrac{T(1,0)^{\mathrm F}}
{L(0,1)\otimes\nabla(1,0)^{\rm F}}}}
\]
\end{Lemma}
\begin{proof}
First note that $(2,1)-(0,2)=(2,-1)$ which is the simple root $\alpha$. Hence, there is a non-trivial $G$-homomorphism from $\nabla(2,1)$ to $\nabla(0,2)$ \cite[5.10 Proposition]{Don}. Now, recall that $\nabla(0,2)$ has simple head $L(0,1)$ and that it is uniserial. Moreover, since $\nabla(2,1)$ appears as a top section in a good filtration of $T(2,1)$ we have that $\nabla(2,1)$ has also simple head $L(0,1)$.  It follows at once that a non-trivial homomorphism from $\nabla(2,1)$ to $\nabla(0,2)$ is surjective and so we have a short exact sequence 
\begin{equation}\label{eq:exact-sequence02}
0 \rightarrow K \rightarrow \nabla(2,1) \rightarrow \nabla(0,2)\rightarrow 0
\end{equation}
Now, comparing the composition factors of $\nabla(2,1)$ and $\nabla(0,2)$, \cite[Appendix C and D]{Do}, we have that the composition factors of $K$ are $L(0,0)$, appearing with multiplicity $2$, and $L(1,0),L(0,1),L(2,0)$ and $L(2,1)$, each appearing with multiplicity $1$. Suppose that $X$ is a submodule of $K$. We consider the short exact sequence 
$$0\rightarrow X\rightarrow K\rightarrow K/X\rightarrow 0$$
and we apply $\Hom_G(L(\lambda),-)$ to this sequence for some composition factor $L(\lambda)$ of $K$. Since $\soc_GX=\soc_GK=L(2,1)$ we obtain the exact sequence 
$$0\rightarrow \Hom_G(L(\lambda),K/X)\rightarrow \Ext^1_G(L(\lambda),X)\rightarrow  \Ext^1_G(L(\lambda),K)$$
We claim that  $\Ext^1_G(L(\lambda),K)=0$. For this, we apply $\Hom_G(L(\lambda),-)$ to  the sequence \eqref{eq:exact-sequence02}, and we  obtain the exact sequence 
$$\Hom_G(L(\lambda),\nabla(2,0))\rightarrow \Ext^1_G(L(\lambda),K)\rightarrow  \Ext^1_G(L(\lambda),\nabla(2,1))$$
Since $\lambda\leq (2,1)$ and $\lambda\neq(0,2)$, we deduce that 
$$\Hom_G(L(\lambda),\nabla(0,2))=\Ext^1_G(L(\lambda),\nabla(2,1))=0$$
and so $\Ext^1_G(L(\lambda),K)=0$. It follows that 
$$\Ext^1_G(L(\lambda),X)=\Hom_G(L(\lambda),K/X)$$
for every composition factor of $K$. In other words, we obtain an analogue of Lemma~\ref{lem:bootstrap} for $K$ that allows us to describe its structure. We proceed directly with this technique.
We have that $\soc_GK=L(2,1)$ and from Lemma~\ref{lem.ext.simples}, among the composition factors of $K$, the simple modules that extend with $L(2,1)$ are $L(0,1)$ and $L(0,0)$. Therefore, we have created the submodule 
 \[
K_1\;=\;
\frac{L(0,1)\oplus L(0,0)}{L(2,1)}
\]
of $K$.
Note that this provides the non spilt short exact sequence
$$0\rightarrow \nabla(0,1)\otimes\nabla(1,0)^F\rightarrow K_1\rightarrow L(0,0)\rightarrow 0$$  
Now, again from the table of Lemma~\ref{lem.ext.simples}, the possible candidates that may appear directly above $K_1$ and inside $K$ are $L(1,0)$, which extends non-trivially with $L(0,0)$, and $L(2,0)$ that extends non-trivially with both $L(0,0)$ and $L(0,1)$.

First, we show that $\Ext^1_G(L(1,0),K_1)=0$. Suppose for contradiction that there is a non-trivial extension. Then $L(1,0)$ can only extend with  $L(0,0)$ and create $\Delta(1,0)$. However, by Lemma~\ref{lem.gcoh} we have that
$$\Ext^1_G(\Delta(1,0),\nabla(0,1)\otimes\nabla(1,0)^F)= H^1(G,\nabla(1,0)\otimes \nabla(0,1)\otimes\nabla(1,0)^F)=0$$
Therefore $L(1,0)$ does not appear in the socle of $K/K_1$ and so $L(2,0)$ must do. 

Now, suppose that $L(2,0)$ does not extend with the $L(0,0)$ from the second socle layer of $K$. This would imply that $K$ contains the $G$-submodule
\[\frac{\dfrac{L(2,0)}{L(0,1)}\oplus \raisebox{-1ex}{$L(0,0)$}}{L(2,1)}\]
However, by Lemma~\ref{lem.coh.simples}(iii) and Lemma~\ref{lem.good.filtrations}(iii), we have that 
\begin{align*}
\Ext^1_G(L(2,0),L(0,1)\otimes\nabla(1,0)^F)
&= H^1(G,L(0,1)\otimes L(1,0)^{\mathrm F}\otimes \nabla(1,0)^{\mathrm F})\\
&= H^0(G,\nabla(1,0)\otimes L(1,0)\otimes \nabla(1,0))=0.
\end{align*}
Therefore, $L(2,0)$ cannot  extend solely  $L(0,1)$ and so it must certainly interact with $L(0,0)$. Hence, we have constructed inside $K$ the $G$-submodule 
\[
K_2\;=\;
\frac{\Delta(1,0)^{\rm F}}{L(0,1)\otimes \nabla(1,0)^{\rm F}}
\]
Now, we have that $K/K_2$ has composition factors $L(0,0)$ and $L(1,0)$. We show that $L(1,0)$ cannot appear in the socle of $K/K_2$. Suppose for contradiction  that it does.  Since $\Ext^1_G(L(1,0),L(2,0))=0$ and $L(1,0)$ does not appear in the second socle layer of $K$, we must have that there  is an extension between the copy of $L(0,0)$ in the second socle layer of $K$ with $L(1,0)$, which in turn forms $\Delta(1,0)$. But as we have seen above,  $\Ext^1_G(\Delta(1,0),\nabla(0,1)\otimes\nabla(1,0)^F)=0$ and so this is impossible. We deduce that $L(1,0)$ does not appear directly above $K_2$ and so we must have that $L(0,0)$ does. In fact, this extends with $\Delta(1,0)^F$ and so we obtain that 
\[
K\;=\;
\frac{\dfrac{L(1,0)}{T(1,0)^{\rm F}}}{L(0,1)\otimes \nabla(1,0)^{\rm F}}
\]
Therefore we have the desired picture for $\nabla(2,1)$.
 \end{proof}

\section{An epimorphism from $\St\otimes \St$ to $T(2,1)$}\label{sec:surj.tilt}
Here, we show that  there is a surjective $G$-homomorphism from $\St\otimes \St$ to the tilting module $T(2,1)$ and that $T(2,2)$ is the unique component in the principal block (i.e., the block of ${\rm mod}(G)$ containing $L(0)$)  of $\St\otimes \St$.

The radical of a module $X$ (i.e., the smallest submodule for which the corresponding quotient is semisimple) will be denoted $\rad X$. We note that it follows (for example from \cite[II, 4.13 Proposition]{J1}) that if $X$ is a $G$-module admitting a Weyl filtration and $Y$ is a $G$-module admitting a good filtration then we have 
$$\dim \Hom_G(X,Y)=\sum_{\lambda\in X^+(T)} (X:\Delta(\lambda))(Y:\nabla(\lambda)).$$

Recall from Section~\ref{sec:sub.in} that that the induced $G$-modules $\nabla(0,2)$ and $\nabla(2,1)$ have both simple head $L(0,1)$. We calculate the dimension of the following spaces of homomorphisms.

\begin{Lemma}\label{lem:dim.St}
We have:
\begin{enumerate}[label = (\roman*), font=\normalfont]
\item $\dim \Hom_G(\St\otimes \St,\nabla (0,2))=2$;
\item $\dim \Hom_G( \St\otimes \St,\nabla(2,1))=3$;
\item $\dim \Hom_G(\St\otimes \St,\rad\nabla(0,2))=1$;
\item $\dim \Hom_G(\St\otimes \St,\rad\nabla(2,1))=2$.
\end{enumerate}
\end{Lemma}

\begin{proof}
Part (i) and (ii) follow directly from the character of $\St \otimes \St$, see Lemma~\ref{lem:ch.stst}.
We now show (iii). From Lemma~\ref{lemma:(0,2)} we have that 
\[\rad\nabla(0,2)\;=\;
\frac{\Delta(1,0)^{\mathrm F}}
{\dfrac{L(1,0)\otimes \nabla(1,0)^{\mathrm F}}
 {L(0,2)}}
\]
and so 
\[\St\otimes\rad\nabla(0,2)\;=\;
\frac{\St\otimes \Delta(1,0)^{\mathrm F}}
{\dfrac{\St\otimes L(1,0)\otimes \nabla(1,0)^{\mathrm F}}
 {\St\otimes L(0,2)}}
\]
Now, $\St\otimes L(0,2)=\nabla(1,3)$ and $\St\otimes \Delta(1,0)^{\mathrm F}=\Delta(3,1)$. Moreover,  $\St\otimes L(1,0)=T(2,1)$, from Lemma~\ref{lem:ch.T(2,1)}, and so $\St\otimes L(1,0)\otimes \nabla(1,0)^{\mathrm F}=T(2,1)\otimes \nabla(1,0)^{\mathrm F}$. Note that  the Steinberg block component (i.e., the component in  the block of ${\rm mod}(G)$ containing $\St$) of $T(2,1)\otimes \nabla(1,0)^{\mathrm F}$ is $0$ and so we deduce that the component in the  Steinberg block of $\St\otimes\rad\nabla(0,2)$ is $\Delta(3,1)\oplus \nabla(1,3)$. It follows that 
$$\dim \Hom_G(\St,\St\otimes \rad\nabla(0,2))=\dim\Hom_G(\St,\Delta(3,1)\oplus \nabla(1,3))=1$$
We now show (iv). From Lemma~\ref{lemma:(2,1)} we have that 
%\[\rad\nabla(2,1)={\mathcal F}rac{\dfrac{\rad\nabla(0,2)}{\dfrac{L(1,0)}{\dfrac{L(1,0)}{T(1,0)^{\mathrm F}}}}{\nabla(0,1)\otimes\nabla(1,0)^F}\]
\[\rad\nabla(2,1)=\frac{\dfrac{\rad\nabla(0,2)}{L(1,0)}}{\dfrac{T(1,0)^{\rm F}}{\nabla(1,0)^F\otimes L(0,1)}}\]
and so 
\[\St\otimes \rad\nabla(2,1)=\frac{\dfrac{\St\otimes\rad\nabla(0,2)}{\St\otimes L(1,0)}}{\dfrac{\St\otimes T(1,0)^{\rm F}}{\St\otimes \nabla(1,0)^{\rm F}\otimes L(0,1)}}\]
Now, $\St\otimes L(0,1)=T(1,2)\oplus \St^2$, from Lemma~\ref{lem:ch.T(1,2)}. Therefore 
$$\St\otimes \nabla(1,0)^F\otimes \nabla(0,1)=T(1,2)\otimes  \nabla(1,0)^F\oplus \nabla(3,1)^2.$$
Note, that  the Steinberg block component of the $G$-module $T(1,2)\otimes  \nabla(1,0)^F$ is $0$. 
Moreover, $\St\otimes T(1,0)^F=T(3,1)$ and $\St\otimes L(1,0)=T(2,1)$. Finally from part (iii) we have that the Steinberg component of $\St\otimes\rad\nabla(0,2)$ is  $\Delta(3,1)\oplus T(1,3)$. Therefore, the Steinberg component  of 
$\St\otimes \rad\nabla(2,1)$ is $T(3,1)\oplus T(1,3)\oplus \Delta(3,1)\oplus \nabla(3,1)^2$. Hence we obtain 
$$\dim \Hom_G(\St,\St\otimes \rad\nabla(2,1))=\dim\Hom_G(\St,T(3,1)\oplus T(1,3)\oplus \Delta(3,1)\oplus \nabla(3,1)^2)=2.$$
\end{proof}

\begin{Proposition}\label{prop:surj}
There is a surjective $G$-homomorphism  from $\St\otimes \St$ to $T(2,1)$.
\end{Proposition}

\begin{proof}
Recall that $T(2,1)$ has simple head the $G$-module $L(0,1)$ and that it has a good filtration with top factor $\nabla(2,1)$. We write $Y$ for the submodule of $T(2,1)$ with quotient $T(2,1)/Y=\nabla(2,1)$. Note that $Y\subseteq \rad T(2,1)$, it has a good filtration  and that $\rad T(2,1)/Y=\rad\nabla(2,1)$.

Now, there is a surjective homomorphism $\St\otimes \St$ from $T(2,1)$ if and only if   \\
$\dim \Hom_G(\St\otimes \St, T(2,1))>\dim \Hom_G(\St\otimes \St, \rad T(2,1))$.  We consider the short exact sequence 
$$0\rightarrow \rad T(2,1)\rightarrow T(2,1)\rightarrow L(0,1)\rightarrow 0$$
We apply $\Hom_G(\St\otimes \St,-)$ and we obtain the exact sequence 
\begin{align*}
0 &\longrightarrow H^0(G,\St\otimes \St\otimes \rad T(2,1))
   \longrightarrow H^0(G,\St\otimes \St\otimes T(2,1))
   \longrightarrow H^0(G, \St\otimes \St\otimes L(0,1))\\
  &\longrightarrow  H^1(G,\St\otimes \St\otimes  \rad T(2,1))
    \longrightarrow 0
    \end{align*}
Now, from Lemma~\ref{lem:ch.T(1,2)},  $\St\otimes L(0,1)=T(1,2)\oplus \St^2$ and so $\dim H^0(G, \St\otimes \St\otimes L(0,1))=2$. It follows that there is a surjective homomorphism from $\St\otimes \St$ to $T(2,1)$ if and only if 
$$\dim  H^1(G,\St\otimes \St\otimes  \rad T(2,1))\leq 1$$
Now, we consider the short exact sequence 
$$0\rightarrow Y\rightarrow \rad T(2,1)\rightarrow \rad \nabla(2,1)\rightarrow0$$
Since $Y$ has a good filtration, we have that $H^i(G,\St\otimes \St\otimes Y)=0$, for $i>0$. Therefore from the corresponding long exact sequence we obtain that 
$$H^1(G,\St\otimes \St\otimes  \rad T(2,1))=H^1(G,\St\otimes \St\otimes  \rad \nabla(2,1))$$
and so we have a surjective homomorphism from $\St\otimes \St$ to $T(2,1)$ if and only if 
$$\dim  H^1(G,\St\otimes \St\otimes  \rad \nabla(2,1))\leq 1$$
Finally, we consider the short exact sequence 
$$0\rightarrow \rad \nabla(2,1)\rightarrow \nabla(2,1)\rightarrow L(0,1)\rightarrow 0$$
We apply $\Hom_G(\St\otimes \St,-)$ and we obtain the exact sequence 
\begin{align*}
0 &\longrightarrow H^0(G,\St\otimes \St\otimes \rad \nabla(2,1))
   \longrightarrow H^0(G,\St\otimes \St\otimes \nabla(2,1))
   \longrightarrow H^0(G, \St\otimes \St\otimes L(0,1))\\
  &\longrightarrow  H^1(G,\St\otimes \St\otimes  \rad \nabla(2,1))
    \longrightarrow 0
    \end{align*}
 Now, we have $\dim H^0(G,\St\otimes \St\otimes \rad \nabla(2,1))=2$ and $\dim H^0(G,\St\otimes \St\otimes \nabla(2,1))=3$, by Lemma~\ref{lem:dim.St}. Moreover, $\dim H^0(G, \St\otimes \St\otimes L(0,1))=2$. It follows that 
 $$\dim H^1(G,\St\otimes \St\otimes  \rad \nabla(2,1))=1$$
 and so we have the desired surjective homomorphism.
 \end{proof}
 
 Since $T(2,1)$ and $\St\otimes\St$ are self-dual modules we also immediately  obtain the following: 
  \begin{Corollary}\label{cor:emb}
There is an embedding of $T(2,1)$ into the $G$-module $\St\otimes \St$.  
\end{Corollary}

 We now proceed with our second claim regarding the decomposition of $\St\otimes \St$ into tilting modules. Since $(2,2)$ is the highest weight of $\St\otimes \St$, we have that $T(2,2)\mid \St\otimes \St$. Now, recall from Lemma~\ref{lem:ch.stst} that $\St\otimes \St$ is injective as a $G_1$-module and we have a decomposition
 $$\St\otimes \St|_{G_1}=Q_1(0)\oplus Q_1(0,1)^2\oplus \St^{16}$$
 Moreover, we have that $T(2,1)$ is injective as $G_1$-module with $T(2,1)|_{G_1}=Q_1(0,1)$. We write $D$ for the component of $\St\otimes \St$ in the principal block and by Lemma~\ref{lem:ch.stst} we deduce that
 $$\St\otimes \St=D\oplus T(3,1)^2$$
and $D$ can be one of the following modules:  $T(2,2)$ or $T(2,2)\oplus T(2,1)$ or $T(2,2)\oplus T(2,1)^2$. In \cite[Theorem 4.1.1]{BNPS2} the authors ruled out the third case and in that way they provided the first counterexample to the Tilting Module conjecture. Here we show that $D=T(2,2)$ and so we obtain the decomposition of $\St\otimes \St$ into tilting modules. 
 
 From now, we write $M$  for the unique non-split extension of $L(2,0)$ by $L(0,1)$, i.e., the $G$-module $M$ fits into the non split short exact sequence 
 $$0\rightarrow L(2,0)\rightarrow M\rightarrow L(0,1)\rightarrow 0$$
 We gather some important results regarding the $G$-module $M$ in the following lemma.
 
 \begin{Lemma}\label{lem:surjM}
For the $G$-module $M$ we have that:
 \begin{enumerate}[label = (\roman*), font=\normalfont]
\item There is a surjective $G$-homomorphism from  $\nabla(0,2)$ to $ M$;
\item There is a surjective $G$-homomorphism from $\nabla(2,1)$ to $M$;
\item There is a surjective $G$-homomorphism  from $\nabla(2,2)$ to $M$;
\item $\dim\Hom_G(\St\otimes \St,M )=1$.
\end{enumerate}
\end{Lemma}

\begin{proof}
Parts (i) and (ii) follow directly from Lemma~\ref{lemma:(0,2)} and Lemma~\ref{lemma:(2,1)}. Part (iii) follows from \cite[Appendix J, J.11]{Do}. Part (iv) can be found in  \cite[\S 3.7]{BNPS2}. 
\end{proof}

Recall from Lemma~\ref{lem:ch.T(2,1)} that $T(2,1)$ has a Weyl filtration 
$$0=E_0\subseteq E_1\subseteq \ldots\subseteq E_4\subseteq E_5=T(2,1),$$
where $E_i/E_{i-1}=\Delta(\lambda^i)$, for $1\leq i\leq 5$, and
$$\lambda^1=(2,1),  \lambda^2=(0,2),  \lambda^3=(3,0),  \lambda^4=(2,0),  \lambda^5=(0,1)$$    
For ease of notation, from now we write $E$ for the tilting module $T(2,1)$.

\begin{Lemma}\label{lem:noM}
For $1\leq i\leq 3$ we have that $\Hom_G(E_i,M)=0$.
\end{Lemma}

\begin{proof}
Suppose for contradiction that the result fails and let $i$ be minimal such that  $\Hom_G(E_i,M)\neq0$.  Therefore $\Hom_G(E_i/E_{i-1},M)\neq0$. However, this is impossible since $E_i/E_{i-1}$ is generated by a vector of weight $\lambda^i$ and $\lambda^i$ is not a weight of $M$.
\end{proof}

\begin{Lemma}\label{lem:nosum}
Let $\phi:T(2,1)\rightarrow D$ be an injective $G$-homomorphism. For $0\leq i\leq 2$, let $V_i=\phi(E_i)$ and let $V=\phi(T(2,1))$. The submodule $V/V_i$ is not a direct summand of $D/V_i$.
\end{Lemma}

\begin{proof}
Suppose, for a contradiction, that it is. Thus we may write $D = X + V$, for some submodule $X$ of $D$ such that $X \cap V = V_i$. Recall that we have an epimorphism $\pi:D\rightarrow \nabla(2,2)$ and so $\nabla(2,2)=\pi(X)+\pi(V)$. Now, the restriction
$\pi|_V : V \to \nabla(2,2)$ is the zero map, since $L(2,2)$ is not a composition factor of $V$. Therefore, by Lemma~\ref{lem:surjM}(iii) we obtain a surjective map $X\to M$. Moreover, by Lemma~\ref{lem:noM} we have that this map is zero on $V_i$ and so $\Hom_G(X/V_i,M)\neq 0$. Similarly, by Lemma~\ref{lem:noM}, an epimorphsim from $V$ to $M$ is zero on $V_i$ and so $\Hom_G(V/V_i,M)\neq 0$. Now the epimorphism 
$$X\oplus V\rightarrow X+V\rightarrow D/V_i$$
has kernel $V_i\oplus V_i$ and so it induces an isomorphism $X/V_i\oplus V/V_i\cong D/V_i$. It follows that $\dim \Hom_G(D/V_i, M)\geq 2$ and so $\dim \Hom_G(D, M)\geq 2$. But this is impossible since $\dim\Hom_G(D,M )=1$, by Lemma~\ref{lem:surjM}(iv).
\end{proof}

 \begin{Proposition}\label{prop:tilt.dec}
 The decomposition of $\St\otimes \St$ into tilting modules is 
 $$\St\otimes \St=T(2,2)\oplus T(3,1)^2$$
 \end{Proposition}
 
 \begin{proof}
Setting $i=0$ in the previous lemma we obtain that $T(2,1)$ is not a direct summand of $D$ and so we obtain that $D=T(2,2)$. This together with our comments above  provide the desired decomposition.
\end{proof}
 
 By Proposition~\ref{prop:surj} and Proposition~\ref{prop:tilt.dec} we deduce that
 
  \begin{Corollary}\label{cor:surj}
 There is a surjective $G$-homomorphism from $T(2,2)$ to $T(2,1)$.
\end{Corollary}

\section{Humphreys-Verma conjecture}\label{sec:struct}
In this section we produce a $G$-structure for the injective  $G_1$-module $Q_1(0)$ and in this way we verify the Humphreys-Verma conjecture for $G_2$ and $p=2$. This module will be obtained as a  sub-quotient of the tilting module $T(2,2)$.  First, we outline our strategy: 

From Corollary~\ref{cor:surj} we have a surjective $G$-homomorphism from  $T(2,2)$ to $T(2,1)$.  We fix a submodule $N$ of $T(2,2)$ such that $T(2,2)/N$ is isomorphic to $T(2,1)$.
Thus we have  $N|_{G_1}=Q_1(0)\oplus Q_1(0,1)$ and  $\soc_G N=L(0,0)\oplus L(0,1)$.  We show that $T(2,1)$ embeds into $N$ and so the $G$-quotient  $Q\coloneqq N/V$, where $V$ is a submodule $N$ isomorphic to $T(2,1)$, has the property $Q|_{G_1}=Q_1(0)$.

\begin{Lemma}\label{lem:ext}
We have 
 \begin{enumerate}[label = (\roman*), font=\normalfont]
\item $\Ext^1_G(\Delta(0,1),N)=0$, and 
 \item $\Ext^1_G(\Delta(2,0),N)=0.$
 \end{enumerate}
 \end{Lemma}
 
 \begin{proof}
 Since $N|_{G_1}=Q_1(0)\oplus Q_1(0,1)$, we have that $H^0(G_1,L(0,1)\otimes N)=L(0,0)$ and $H^1(G_1,L(0,1)\otimes N)=0$. Therefore, by the 5-term exact sequence we obtain 
 $$\Ext^1_G(L(0,1),N)=H^1(G,L(0,1)\otimes N)=H^1(G,L(0,0))=0$$
Now for the second part, recall that $\Delta(2,0)$ has a filtration with sections $\Delta(1,0)^F, L(0,1)$ and $L(1,0)$. It is enough to show then that $N$ does not extend non-trivially with any of these modules. We have already dealt with $L(0,1)$ in the previous step. Moreover,  using the $5$-term exact sequence, we immediately get that $\Ext^1_G(L(1,0),N)=0$. Now, note that  $H^0(G_1,\nabla(1,0)^F\otimes N)=\nabla(1,0)^F$ and   $H^1(G_1,\nabla(1,0)^F\otimes N)=0$. Therefore, again  by the $5$-term exact sequence, we obtain  
 $$\Ext^1_G(\Delta(1,0)^F,N)=H^1(G,\nabla(1,0)^F\otimes N)=H^1(G,\nabla(1,0))=0$$
 It follows that $\Ext^1_G(\Delta(2,0),N)=0$.
 \end{proof}

Recall the filtration $0=E_0\subseteq E_1\subseteq \ldots\subseteq E_4\subseteq E_5=T(2,1)$ of $T(2,1)$ from the previous section.  We note that each $E_i$ is characteristic in $E$, in the sense that it is preserved by any endomorphism. To see this we note that, for $1\leq i\leq 5$, the quotient $E/E_{i-1}$ has unique highest weight $\lambda^i$, and this occurs with multiplicity one. Thus if $v_i\in E\backslash E_{i-1}$ is any vector of weight $\lambda^i$, then $E_i$ is generated by $E_{i-1}$ and $v_i$. Now if  $\theta$ is an endomorphism of $E$ then $\theta(E_i)$ is generated by $\theta(E_{i-1})$ and $\theta(v_i)$.  We can assume inductively that $\theta(E_{i-1})\subseteq E_{i-1}$ and then either $\theta(v_i)\in E_{i-1}$ or is a non-zero vector in $E\backslash E_{i-1}$ of weight $\lambda^i$. Thus $\theta(v_i)\in E_i$ and $\theta(E_i)\subseteq E_i$.

We have the following key result.
 
\begin{Proposition}\label{prop:emb.N}
 For $0\leq i\leq 2$, any embedding of the $G$-module $E_i$ into $N$ extends to an embedding of $E_{i+1}$ into $N$.
 \end{Proposition}
 
 \begin{proof}
We write $D$ for $T(2,2)$ and $E$ for $T(2,1)$. For $0\leq i\leq 2$, let $\phi:E_i\to D$ an embedding with image $\phi(E_i)=V_i$, say, contained in $N$. The quotient $E/E_i$ has a Weyl filtration and so $\Ext^1_G(E/E_i,D)=0$. Hence, the map 
$$\Hom_G(E,D)\rightarrow \Hom_G(E,E_i)$$
is a surjection and so $\phi$ extends to a homomorphism $\psi: E\to D$. If $i>0$ then this map is injective, since $E$ has simple $G$-socle $L(0,1)$ that is contained in $E_i$. On the other hand, if $i=0$, then we can (and we do) choose  $\psi$ to be an injective homomorphism by Corollary~\ref{cor:emb}. We set   $V_{i+1}\coloneqq \psi (E_{i+1})$. Now we have that $\psi(E_i)=\phi(E_i)=V_i$ and so ${\bar \psi}$  induces a $G$-homomorphism $\bar{\psi}:E/E_i\rightarrow D/V_i$. 

Now let $\pi:D\to E$ be any surjective $G$-homomorphism with kernel $N$.  The composite $\pi\circ \phi$ defines an endomorphism of $E$ and, since $E_i$ is characteristic, we have $\pi(\phi(E_i))\subseteq E_i$, i.e., $\pi(V_i)\subseteq E_i$. Thus we get an induced map $\bar{\pi}:D/V_i\to E/E_i$. 

We consider the composite map $\theta\coloneqq \bar{\psi}\circ  \bar{\pi}\in \End_G(D/V_i)$. We set $V\coloneqq \psi (E)$. We have $\im \theta \subseteq V/V_i$, and by Fitting's Lemma, $\bigcap_{i\geq1}\im \theta^i$ is a direct summand of $D/V_i$. Now, by Lemma~\ref{lem:nosum}, it is not $V/V_i$ and since $V/V_i$ is indecomposable (it has simple head  $L(0,1)$), it is not a proper summand of $V/V_i$. Thus $\bar{\psi}\circ  \bar{\pi}$ is nilpotent. But then, so is $\theta=\bar{\pi}\circ \bar{\psi}$.  Now $E_{i+1}$ is generated by $E_i$ and a non-zero vector $v_{i+1}\in E_{i+1}\backslash E_i$   of weight $\lambda^{i+1}$.   Moreover,  ${\mathbb F} v_{i+1} + E_i$ is the $\lambda^{i+1}$ weight space of $E/E_i$. Thus we must have $\theta(v_{i+1}+E_i)\in  {\mathbb F} v_{i+1} +E_i$, i.e., $\theta(v_{i+1})+E_i=cv_{i+1}+E_i$, for some scalar $c\in {\mathbb F}$. But, by nilpotency of $\theta$ we must have $c=0$, i.e. $\theta(v_{i+1}+E_i)=E_i$ and $\theta(E_{i+1}/E_i)=0$. 
   Hence, $\bar{\psi}(E_{i+1}/E_i)\subseteq \ker\bar{\pi}$, i.e., $V_{i+1}/V_i\subseteq N/V_i$ and so $V_{i+1}\subseteq N$.
 \end{proof}
 
 \begin{Proposition}\label{prop:emb.N} The kernel of any $G$-module epimorphism from $T(2,2)$ to $T(2,1)$ contains a copy of $T(2,1)$. 
 \end{Proposition}
 \begin{proof}  We write $E$ for $T(2,1)$.  Let $N$ be the kernel of an epimorphism from $D=T(2,2)$ to $E$.
By Proposition~\ref{prop:emb.N}, we have an embedding of $E_3$ into $N$. Moreover, by Lemma~\ref{lem:ext} we have that $\Ext^1_G(E/E_3,N)=0$. Therefore, the restriction map 
 $$\Hom_G(E,N)\rightarrow \Hom_G(E_3,N)$$
 is surjective. Since $E$ has simple $G$-socle $L(0,1)$ that is contained in $E_3$, we deduce at once that there is an embedding of $E$ into $N$.
 \end{proof}
 
 By Proposition~\ref{prop:surj}, Corollary~\ref{cor:surj} and Proposition~\ref{prop:emb.N} we obtain the main theorem of this paper.
  \begin{Theorem}\label{the:struct}
 We have a $G$-structure on $Q_1(0)$, namely $N/V$, where $N$ is the kernel of a surjective $G$-homomorphism $T(2,2)\rightarrow T(2,1)$ and $V$ is a $G$-submodule of $N$ isomorphic to $T(2,1)$.
 \end{Theorem}

\end{document}